\documentclass{scrartcl}
\usepackage{latexsym} 
\usepackage{amsmath}
\usepackage{amsthm}
\usepackage{amsfonts}
\usepackage{amssymb}
\usepackage{amsxtra}
\usepackage{amscd}
\usepackage[shortlabels]{enumitem}
\usepackage{mathrsfs}
\usepackage{hyperref}
\usepackage{mathtools}

\numberwithin{equation}{section} 

\newenvironment{pdeq}{ \left\{ \begin{aligned}}{\end{aligned}\right.}

\newcommand{\eqrefsub}[2]{\ensuremath{\eqref{#1}_{#2}}}
\newcommand{\np}[1]{(#1)}
\newcommand{\nb}[1]{[#1]}
\newcommand{\bp}[1]{\big(#1\big)}
\newcommand{\bb}[1]{\big[#1\big]}
\newcommand{\Bp}[1]{\bigg(#1\bigg)}

\newcommand{\Bb}[1]{\bigg[#1\bigg]}
\newcommand{\calp}{{\mathcal P}}

\newcommand{\calt}{{\mathcal T}}

\newcommand{\R}{\mathbb{R}}

\newcommand{\Z}{\mathbb{Z}}
\newcommand{\N}{\mathbb{N}}

\DeclareMathOperator{\e}{e}

\DeclareMathOperator{\Div}{div}

\DeclareMathOperator{\supp}{supp}

\newcommand{\set}[1]{\ensuremath{\{#1\}}}

\newcommand{\setL}[1]{\ensuremath{\biggl\{#1\biggr\}}}

\newcommand{\setc}[2]{\ensuremath{\{#1\ \vert\ #2\}}}

\newcommand{\ball}{\mathrm{B}}

\newcommand{\proj}{\calp}
\newcommand{\projcompl}{\calp_\bot}

\newcommand{\grp}{G}
\newcommand{\dualgrp}{\widehat{G}}

\newcommand{\torus}{{\mathbb T}}

\newcommand{\idmatrix}{I}
\newcommand{\Rn}{{\R^n}}

\newcommand{\grad}{\nabla}

\newcommand{\D}{{\mathrm D}}

\newcommand{\dxt}{{\mathrm d}(x,t)}
\newcommand{\dtx}{{\mathrm d}(t,x)}

\newcommand{\dtheta}{{\mathrm d}\theta}
\newcommand{\dtau}{{\mathrm d}\tau}

\newcommand{\ds}{{\mathrm d}s}
\newcommand{\dt}{{\mathrm d}t}

\newcommand{\dy}{{\mathrm d}y}

\newcommand{\TDR}{\mathscr{S^\prime}}

\newcommand{\ft}[1]{\widehat{#1}}

\newcommand{\FT}{\mathscr{F}}
\newcommand{\iFT}{\mathscr{F}^{-1}}

\newcommand{\norm}[1]{\lVert#1\rVert}

\newcommand{\snorm}[1]{{\lvert #1 \rvert}}
\newcommand{\snorml}[1]{{\bigl\lvert #1 \big\rvert}}
\newcommand{\snormL}[1]{{\Bigl\lvert #1 \Big\rvert}}

\newcommand{\LR}[1]{\mathrm{L}^{#1}}

\newcommand{\LRloc}[1]{\mathrm{L}^{#1}_{\mathrm{loc}}} 
\newcommand{\CR}[1]{\mathrm{C}^{#1}}  
\newcommand{\CRi}{\CR \infty}
\newcommand{\CRci}{\CR \infty_0}

\newcommand{\WSR}[2]{\mathrm{W}^{#1,#2}}

\newcommand{\DSRN}[2]{\mathrm{D}^{#1,#2}_0}

\newcommand{\DSRNsigma}[2]{\mathrm{D}^{#1,#2}_{0,\sigma}}

\newcommand{\CRcisigma}{\CR{\infty}_{0,\sigma}}

\newcommand{\vvel}{v}
\newcommand{\vpres}{p}

\newcommand{\wvel}{w}

\newcommand{\uvel}{u}

\newcommand{\upres}{\mathfrak{p}}

\newcommand{\Uvel}{U}

\newcommand{\wakefct}[1]{s(#1)}
\newcommand{\fundsolvel}{\varGamma^\rey}
\newcommand{\fundsolveljl}{\varGamma^\rey_{j\ell}}

\newcommand{\fundsolcompl}{\varGamma^\bot}

\newcommand{\fundsolvelss}{\fundsolvel_0}
\newcommand{\fundsolvelssjl}{\fundsolvel_{0,j\ell}}

\newcommand{\fundsolvelpp}{\fundsolvel_\perp}

\newcommand{\tin}{\text{in }}

\newcommand{\tfor}{\text{for }}

\newcommand{\half}{\frac{1}{2}}

\renewcommand{\epsilon}{\varepsilon}
\newcommand{\rey}{\lambda}

\newcommand{\per}{\calt}

\newcommand{\perf}{\frac{2\pi}{\per}}
\newcommand{\perft}{\tfrac{2\pi}{\per}}

\newcommand{\eone}{\e_1}

\newcommand{\restterm}{\mathscr{R}}

\newcommand{\cutoff}{\chi}

\newcommand{\onedist}{1}
\newcommand{\change}[1]{}

\theoremstyle{plain}
\newtheorem{thm}{Theorem}[section]
\newtheorem{defn}[thm]{Definition}
\newtheorem{lem}[thm]{Lemma}
\newtheorem{prop}[thm]{Proposition}

\theoremstyle{remark}
\newtheorem{rem}[thm]{Remark}

\begin{document}
\title{On the spatially asymptotic structure of time-periodic solutions to the Navier--Stokes equations}

\author{Thomas Eiter%
}

\maketitle

\begin{abstract}
The asymptotic behavior of weak time-periodic solutions to the Navier--Stokes equations
with a drift term 
in the three-dimensional whole space 
is investigated. 
The velocity field is decomposed into a time-independent and a remaining part,
and separate asymptotic expansions are derived
for both parts and their gradients.
One observes that the behavior at spatial infinity is determined 
by the corresponding Oseen fundamental solutions.
\end{abstract}


\section{Introduction}
\label{sec:Introduction}

We study the behavior for $\snorm{x}\to\infty$ of
time-periodic solutions to the Navier--Stokes equations 
\begin{equation}\label{sys:NStp_Decay}
\begin{pdeq}
\partial_t\uvel-\Delta\uvel-\rey\partial_1\uvel+\uvel\cdot\grad\uvel+\grad\upres&=f
&&\tin\torus\times\R^3, \\
\Div\uvel&=0 
&&\tin\torus\times\R^3, \\
\lim_{\snorm{x}\to\infty}\uvel(t,x)&= 0
&&\tfor t\in\torus,
\end{pdeq}
\end{equation}
which model the flow of a viscous incompressible fluid.
Here $f\colon\torus\times\R^3\to\R^3$ is an external force,
and $\uvel\colon\torus\times\R^3\to\R^3$ and $\upres\colon\torus\times\R^3\to\R$ 
denote velocity and pressure fields of the fluid flow. 
The torus group $\torus\coloneqq\R/\per\Z$ serves as time axis
and encodes that all involved functions are time periodic 
with prescribed period $\per>0$.
In this paper, we consider the case $\rey\neq0$,
which models a non-vanishing inflow velocity $\rey\eone$ at infinity. 
Asymptotic properties in the case $\rey=0$ are different and shall not be treated here.

For $\rey\neq0$ the pointwise decay
of time-periodic solutions to \eqref{sys:NStp_Decay}
was studied by \textsc{Galdi} and \textsc{Sohr} \cite{GaldiSohr2004}
and by \textsc{Galdi} and \textsc{Kyed} \cite{GaldiKyed_TPSolNS3D_AsymptoticProfile}.
By \cite{GaldiKyed_TPSolNS3D_AsymptoticProfile} a weak solution $\uvel$ to \eqref{sys:NStp_Decay}
satisfies
\begin{equation}\label{eq:VelocityExpansion_Intro}
\uvel(t,x)
=\fundsolvelss(x)\cdot\int_\torus\int_{\R^3}f(s,y)\,\dy\ds 
+ \restterm(t,x),
\end{equation}
where $\fundsolvelss$ is the fundamental solution to the steady-state Oseen system
\begin{equation}\label{sys:NSlinss_TPFS}
\begin{pdeq}
-\Delta\vvel - \rey\partial_1\vvel +\grad\vpres &= f && \tin\R^3, \\
\Div\vvel &=0 && \tin\R^3,
\end{pdeq}
\end{equation}
and the remainder term satisfies $\snorm{\restterm(t,x)}\leq C\snorm{x}^{-3/2+\varepsilon}$.
In particular, 
\eqref{eq:VelocityExpansion_Intro} shows that 
the asymptotic behavior of the velocity field $\uvel$ is, in general, determined by
the steady-state Oseen fundamental solution $\fundsolvelss$. 
Moreover, \eqref{eq:VelocityExpansion_Intro} coincides with 
the anisotropic expansion of weak solutions to the corresponding steady-state problem,
which is due to 
\textsc{Finn} \cite{Finn1959b, Finn1965a}, \textsc{Babenko} \cite{babenko1973} 
and \textsc{Galdi} \cite{GaldiBookNew}
and may be seen as a special case of the time-periodic setting.

The main theorem of this paper, Theorem \ref{thm:VelocityExpansion} below, extends 
the results from \cite{GaldiKyed_TPSolNS3D_AsymptoticProfile}
in several ways.
Firstly, we improve the pointwise estimate of $\restterm(t,x)$
in such a way that it reflects the anisotropic structure of the solution.
Secondly, we derive an asymptotic expansion for $\grad\uvel$
by establishing pointwise estimates of $\grad\restterm(t,x)$.
Thirdly, we decompose $\uvel$ into its time mean over one period $\proj\uvel$ 
and a time-periodic remainder $\projcompl\uvel=\uvel-\proj\uvel$,
for which we derive separate asymptotic expansions. 
We shall observe that the asymptotic properties of the steady-state part $\proj\uvel$ 
are governed by the steady-state fundamental solution $\fundsolvelss$,
while those of the purely periodic part $\projcompl\uvel$
are determined by $\fundsolvelpp$,
the (faster decaying) purely periodic part of the fundamental solution $\fundsolvel$ 
to the time-periodic Oseen system 
\begin{equation}\label{sys:NSlintp_TPFS}
\begin{pdeq}
\partial_t\uvel-\Delta\uvel - \rey\partial_1\uvel +\grad\upres &= f 
&& \tin\torus\times\R^3, \\
\Div\uvel &=0 && \tin\torus\times\R^3.
\end{pdeq}
\end{equation}
In particular, this shows that the purely periodic part $\projcompl\uvel$ 
decays faster than the steady-state part $\proj\uvel$ as $\snorm{x}\to\infty$.

This paper is structured as follows. 
After introducing the basic notation 
in Section \ref{sec:Notation},
we recall the fundamental solution to the time-periodic Oseen equations
and collect related results
in Section \ref{sec:TPFundSol}.
In Section \ref{sec:ProofMainResults} we present and prove 
our main theorems.

\section{Notation}
\label{sec:Notation}

In general, 
we denote points in $\torus\times\R^3$ by $(t,x)$ and call $t\in\torus$ time variable and $x\in\R^3$ 
spatial variable, respectively.
For a sufficiently regular function $\uvel\colon\torus\times\R^3\to\R^3$ 
we write $\partial_j\uvel\coloneqq\partial_{x_j}\uvel$,
and we set $\Delta\uvel\coloneqq\partial_j\partial_j\uvel$ 
and $\Div\uvel\coloneqq\partial_j\uvel_j$.
As in this definition, we use Einstein's summation convention frequently.
If $\Uvel\colon\torus\times\R^3\to\R^{3\times3}$ is matrix valued,
the vector field $\Div\Uvel$ is defined by $(\Div\Uvel)_j=\partial_k\Uvel_{jk}$.

For $R>r>0$ and $x\in\R^3$ we set $\ball_R(x)\coloneqq\setc{y\in\R^3}{\snorm{x-y}<R}$,
$\ball^R(x)\coloneqq\setc{y\in\R^3}{\snorm{x-y}>R}$
and $\ball_{r,R}(x)\coloneqq \ball^{r}(x)\cap\ball_{R}(x)$.
If $x=0$, we simply write $\ball_R\coloneqq\ball_R(0)$, $\ball^R\coloneqq\ball^R(0)$
and $\ball_{r,R}\coloneqq\ball_{r,R}(0)$.
For vectors $a,b\in\R^3$ their tensor product $a\otimes b$ 
is defined by $(a\otimes b)_{jk}=a_jb_k$.

By $\LR{q}(\Omega)$ and $\WSR{k}{q}(\Omega)$ 
we denote classical Lebesgue and Sobolev spaces,
and we set
\[
\CRcisigma(\R^3)\coloneqq\setc{\varphi\in\CRci(\R^3)^3}{\Div\varphi=0},
\qquad \DSRNsigma{1}{2}(\R^3)\coloneqq\overline{\CRcisigma(\R^3)}^{\norm{\grad\cdot}_2}.
\]

Observe that $\grp\coloneqq\torus\times\R^3$ is a locally compact abelian group
and that its dual group can be identified with $\dualgrp=\Z\times\R^3$,
the elements of which we denote by $(k,\xi)\in\Z\times\R^3$.
We equip the group $\torus$ with the normalized Haar measure given by
\[
\forall f\in\CR{}(\torus):\qquad
\int_\torus f(t)\,\dt=\frac{1}{\per}\int_0^\per f(t)\,\dt,
\]
and $\grp$ with the corresponding product measure.
Moreover, $\FT_\grp$ denotes the Fourier transform on $\grp$
with inverse $\iFT_\grp$.
Then $\FT_\grp$ is an isomorphism $\FT_\grp\colon\TDR(\grp)\to\TDR(\dualgrp)$,
where $\TDR(\grp)$ is the space of tempered distributions on $\grp$,
which was introduced by \textsc{Bruhat} \cite{Bruhat61}; see also \cite{EiterKyed_tplinNS_PiFbook}.
Moreover, for $f\colon\torus\times\R^3\to\R$
we set
\[
\proj f (x)\coloneqq \int_\torus f(t,x)\,\dt, 
\qquad \projcompl f\coloneqq f-\proj f
\]
such that $f=\proj f+\projcompl f$.
Since $\proj f$ is time independent, we call 
$\proj f$ the \emph{steady-state} part and $\projcompl f$ the \emph{purely periodic} part of $f$.
A straightforward calculation shows
\[
\proj f = \iFT_\grp \bb{\delta_\Z(k) \FT_\grp\nb{f}}, 
\qquad
\projcompl f = \iFT_\grp \bb{\np{1-\delta_\Z(k)} \FT_\grp\nb{f}}, 
\]
where $\delta_\Z$ is the delta distribution on $\Z$.

By the letter $C$ we denote generic positive constants. 
In order to specify the dependence of $C$ on quantities $a, b, \ldots$, 
we write $C(a, b, \ldots)$.

\section{The time-periodic fundamental solution}
\label{sec:TPFundSol}

In this section, we consider a fundamental solution $\fundsolvel$ to the time-periodic problem 
\eqref{sys:NSlintp_TPFS} such that the velocity field is given by $\uvel=\fundsolvel\ast f$,
where the convolution is taken with respect to the group $\grp=\torus\times\R^3$.
Such a fundamental solution was recently introduced in 
\cite{GaldiKyed_TPSolNS3D_AsymptoticProfile, EiterKyed_etpfslns}
and is given by 
\begin{equation}
\fundsolvel \coloneqq \fundsolvelss\otimes \onedist_{\torus} + \fundsolvelpp,
\label{eq:tpfundsol_decompVel}
\end{equation} 
where 
\begin{align}
&\fundsolvelss\colon\R^3\setminus\set{0}\to\R^{3\times3}, \quad 
\fundsolvelssjl(x)
\coloneqq\frac{1}{4\pi\rey}\bb{\delta_{j\ell}\Delta-\partial_j\partial_\ell}
\int_0^{\wakefct{\rey x} /2} \frac{1-\e^{-\tau}}{\tau}\,\dtau, 
\label{eq:OseenFundSolss_3d}\\
&\fundsolvelpp \coloneqq \iFT_\grp\Bb{ 
\frac{1-\delta_\Z(k)}{\snorm{\xi}^2 + i(\perf k - \rey \xi_1)}\,
\Bp{\idmatrix - \frac{\xi\otimes\xi}{\snorm{\xi}^2}}}.
\label{eq:tpfundsol_deffundsolcompl}
\end{align}
Here the symbol $\onedist_{\torus}$ denotes the constant $1$ distribution on $\torus$, and
\[
\wakefct{x}\coloneqq \snorm{x}+x_1.
\]
The function $\fundsolvelss$ is the fundamental solution to the steady-state Oseen problem
\eqref{sys:NSlinss_TPFS};
see \cite[Section VII.3]{GaldiBookNew}. 
Its anisotropic behavior is reflected by the pointwise estimates
\begin{equation}
\forall \alpha\in\N_0^3\ \forall \epsilon>0\ \exists C>0\ \forall \snorm{x}\geq \epsilon:\quad  
\snorm{\D^\alpha \fundsolvelss(x)} \leq  C\bb{\snorm{x}\np{1+\wakefct{\rey x}}}
^{-1-\frac{\snorm{\alpha}}{2}};
\label{est:fundsolss_Decay}
\end{equation}
see \cite[Lemma 3.2]{Farwig_habil}.
The examination of convolutions of $\fundsolvelss$ with functions satisfying similar estimates 
was carried out by 
\textsc{Farwig} \cite{Farwig_habil,FarwigOseenAnisotropicallyWeightedSob} 
in dimension $n=3$, 
and later by 
\textsc{Kra\v cmar}, \textsc{Novotn\'y} and \textsc{Pokorn\'y}
\cite{KracmarNovotnyPokorny2001} in the general $n$-dimensional case.
The following theorem collects some of their results.

\begin{thm}\label{thm:ConvFundsolss}
Let $A\in[2,\infty)$, $B\in[0,\infty)$ and $g\in\LR{\infty}(\R^3)$ 
such that $\snorm{g(x)}\leq M\np{1+\snorm{x}}^{-A}\np{1+\wakefct{x}}^{-B}$. 
Then there exists
$ C= C(A,B,\rey)>0$
with the following properties:
\begin{enumerate}
\item
If $A+\min\set{1,B}>3$, then
\[
\snorml{\snorm{\fundsolvelss}\ast g (x)} 
\leq  C M \bb{\np{1+\snorm{x}}\bp{1+\wakefct{\rey x}}}^{-1}.
\]
\item
If $A+\min\set{1,B}>3$ and $A+B\geq7/2$, then
\[
\snorml{\snorm{\grad\fundsolvelss}\ast g (x)} 
\leq  C M \bb{\np{1+\snorm{x}}\bp{1+\wakefct{\rey x}}}^{-3/2}.
\]
\item
If $A+\min\set{1,B}=3$ and $A+B\geq7/2$, then
\[
\snorml{\snorm{\grad\fundsolvelss}\ast g (x)} 
\leq  C M \bb{\np{1+\snorm{x}}\bp{1+\wakefct{\rey x}}}^{-3/2}\max\set{1,\log\snorm{x}}.
\]
\item
If $A+B< 3$, then
\[
\snorml{\snorm{\grad\fundsolvelss}\ast g (x)} 
\leq  C M \np{1+\snorm{x}}^{-\np{A+B}/2}
\bp{1+\wakefct{\rey x}}^{-\np{A+B-1}/2}.
\]
\end{enumerate}
\end{thm}

\begin{proof}
These are special cases of \cite[Theorems 3.1 and 3.2]{KracmarNovotnyPokorny2001}.
\end{proof}

In order to derive a similar result to control convolutions 
with the purely periodic part $\fundsolvelpp$, 
we recall the following theorem established in \cite{EiterKyed_etpfslns}.

\begin{thm}\label{thm:tpfundsol}
The purely periodic velocity fundamental solution $\fundsolvelpp$ satisfies
\begin{align}
&\forall q\in\Bp{1,\frac{5}{3}}:\quad \fundsolcompl\in\LR{q}(\grp)^{3\times 3},
\label{est:tpfundsol_ComplSummability}\\
&\forall q\in\bigg[1,\frac{5}{4}\bigg):\quad 
\partial_j\fundsolcompl\in\LR{q}(\grp)^{3\times 3}\quad(j=1,2,3),
\label{est:tpfundsol_ComplSummabilityGradient}\\
&\forall \alpha \in \N_0^3 \ 
\forall r\in [1,\infty)\ \forall\epsilon>0\ \exists C>0\ \forall \snorm{x}\geq \epsilon:\  
\norm{\D_x^\alpha \fundsolcompl(\cdot,x)}_{\LR{r}(\torus)} \leq C\snorm{x}^{-3-\snorm{\alpha}}.
\label{est:tpfundsol_ComplPointwiseEst}
\end{align}
\end{thm}

\begin{proof}
See \cite[Theorem 1.1]{EiterKyed_etpfslns}.
\end{proof}

From these properties we conclude the following theorem.

\begin{thm}\label{thm:ConvFundsolpp}
Let $A\in(0,\infty)$ and $g\in\LR{\infty}(\torus\times\R^3)$ 
such that $\snorm{g(t,x)}\leq M\np{1+\snorm{x}}^{-A}$.
Then for any $\varepsilon>0$ there exists 
$C= C(A,\rey,\per,\varepsilon)>0$
such that
\begin{equation}
\forall \snorm{x}\geq \epsilon: \qquad
\snorml{\snorm{\grad\fundsolvelpp}\ast_{\grp} g (t,x)} 
\leq  C M 
\np{1+\snorm{x}}^{-\min\set{A,4}}
\label{est:ConvFundsolpp_Grad}
\end{equation}
and, if $A>3$,
\begin{equation}
\forall \snorm{x}\geq \epsilon: \qquad
\snorml{\snorm{\fundsolvelpp}\ast_{\grp} g (t,x)} 
\leq C M
\np{1+\snorm{x}}^{-3}.
\label{est:ConvFundsolpp}
\end{equation}
\end{thm}

\begin{proof}
Let us focus on the derivation of \eqref{est:ConvFundsolpp}.
Let $x\in\R^3$, $\snorm{x}\geq\epsilon$ and set $R\coloneqq\snorm{x}/2$. 
Then we have 
\[
\snorml{\snorm{\fundsolvelpp}\ast_\grp g (t,x)}\leq M\np{I_1+I_2+I_3}
\]
with
\[
\begin{aligned}
I_1
&= \int_\torus\int_{\ball_R}\snorm{\fundsolvelpp(t-s,x-y)}\,\np{1+\snorm{y}}^{-A}\,\dy\ds,
\\
I_2
&= \int_\torus\int_{\ball^{4R}} \snorm{\fundsolvelpp(t-s,x-y)}\,\np{1+\snorm{y}}^{-A}\,\dy\ds,
\\
I_3
&= \int_\torus\int_{\ball_{R,4R}}\snorm{\fundsolvelpp(t-s,x-y)}\,\np{1+\snorm{y}}^{-A}\,\dy\ds.
\end{aligned}
\]
We estimate these terms separately.
Since $\snorm{y}\leq R$ implies $\snorm{x-y}\geq\snorm{x}-\snorm{y}\geq\snorm{x}/2=R\geq\epsilon/2$,
we can use \eqref{est:tpfundsol_ComplPointwiseEst} to estimate
\[
I_1
\leq  C \int_{\ball_R} \snorm{x-y}^{-3} \np{1+\snorm{y}}^{-A}\,\dy 
\leq  C  \snorm{x}^{-3} \int_{\R^3}\np{1+\snorm{y}}^{-A}\,\dy
\leq  C  \snorm{x}^{-3}.
\]
For $I_3$ we note that $\snorm{y}\geq 4R$ implies 
$\snorm{x-y}\geq\snorm{y}-\snorm{x}\geq\snorm{y}-\snorm{y}/2=\snorm{y}/2\geq 2R \geq \epsilon$. 
Therefore, \eqref{est:tpfundsol_ComplPointwiseEst} yields
\[
I_2
\leq  C \int_{\ball^{4R}} \snorm{x-y}^{-3} \np{1+\snorm{y}}^{-A}\,\dy 
\leq  C  \int_{\ball^{4R}}\snorm{y}^{-3}\snorm{y}^{-A}\,\dy 
\leq  C  \snorm{x}^{-A}.
\]
Furthermore, H\"older's inequality with $q\in(1,\frac{5}{3})$ and $q'=q/(q-1)$
implies
\[
I_3
\leq \snorm{x}^{-A} \Bp{\int_\torus\int_{\ball_{R,4R}}1 \,\dy\ds}^{1/q'}
\norm{\fundsolvelpp}_{q}
\leq  C  \snorm{x}^{-A} \snorm{x}^{3-\frac{3}{q}}
\]
in virtue of \eqref{est:tpfundsol_ComplSummability}.
We now choose $q\in(1,\frac{5}{3})$ so small that $-A+3-\frac{3}{q}<-3$. 
Collecting these estimates, we obtain \eqref{est:ConvFundsolpp}. 
A proof of \eqref{est:ConvFundsolpp_Grad} can be given in a similar way. 
\end{proof}

The next lemma can be used to conclude asymptotic expansions in the 
linear case, where the velocity field is given by $\uvel=\fundsolvel\ast f$.

\begin{lem}\label{lem:LeadingTerm}
Let $\rey\neq 0$ and $f\in\CRci(\torus\times\R^3)$
with $\supp f\subset\torus\times\ball_{R_0}$. 
Let $\snorm{\alpha}\leq 1$.
Then
\begin{align}
\snorml{\D_x^\alpha\fundsolvelss\ast \proj f(x)}
&\leq C \bb{\np{1+\snorm{x}}\bp{1+\wakefct{\rey x}}}^{-1-\snorm{\alpha}/2}, 
\label{est:LeadingTermss}
\\
\snorml{\D_x^\alpha\fundsolvelpp\ast \projcompl f(t,x)}
&\leq C \np{1+\snorm{x}}^{-3-\snorm{\alpha}}, 
\label{est:LeadingTermpp}
\end{align}
and for $\snorm{x}\geq 2R_0$ we have
\begin{align}
&\snormL{\D^\alpha_x\fundsolvelss\ast\proj f(x)
-\D^\alpha_x\fundsolvelss(x) \cdot\int_{\R^3} \proj f(y)\,\dy}
\leq C\bb{\snorm{x}\np{1+\wakefct{\rey x}}}^{-3/2-\snorm{\alpha}/2},
\label{est:LeadingTermDiffss}
\\
&\snormL{\D^\alpha_x\fundsolvelpp\ast\projcompl f(t,x)
-\Bp{\D^\alpha_x\fundsolvelpp(\cdot,x)\ast_\torus \int_{\R^3} \projcompl f(\cdot,y)\,\dy}(t)}
\leq C\snorm{x}^{-4-\snorm{\alpha}}.
\label{est:LeadingTermDiffpp}
\end{align}
\end{lem}

\begin{proof}
Estimates \eqref{est:LeadingTermss} and \eqref{est:LeadingTermpp}
directly follow from Theorem \ref{thm:ConvFundsolss} and
Theorem \ref{thm:ConvFundsolpp}.
By the mean value theorem, we further have
\[
\snormL{\D^\alpha_x\fundsolvelss\ast\proj f(x)
-\D^\alpha_x\fundsolvelss(x) \int_{\R^3} \proj f(y)\,\dy}
\leq\int_{\ball_R} \int_0^1  \snorm{y} \snorm{\grad \D^\alpha_x\fundsolvelss(x-\theta y)}
\snorm{\proj f(y)}\,\dtheta\dy.
\]
Since $\snorm{y}\leq R_0\leq\snorm{x}/2$ implies
\[
\begin{aligned}
\snorm{x-\theta y}
&\geq\snorm{x}-\theta\snorm{y}
\geq\snorm{x}/2\geq R_0,
\\
\np{1+2\snorm{\rey}R_0}\np{1+\wakefct{\rey \np{x-\theta y}}}
&\geq 1+2\snorm{\rey}R_0+\wakefct{\rey \np{x-\theta y}}
\geq 1+\wakefct{\rey x},
\end{aligned}
\]
estimate \eqref{est:fundsolss_Decay} finally leads to
\eqref{est:LeadingTermDiffss}.
Using \eqref{est:tpfundsol_ComplPointwiseEst} instead of \eqref{est:fundsolss_Decay},
we conclude \eqref{est:LeadingTermDiffpp} in the same way.
\end{proof}

The following auxiliary result treats convolutions of functions 
with anisotropic decay.

\begin{lem}\label{lem:ConvEst_WakeFct_Two}
Let $A\in(-2,2]$, $B\in(1,2]$.
Then there exists $C=C\np{A,B}>0$ such that 
for all $x\in\R^3\setminus\set{0}$ it holds
\[
\begin{aligned}
\int_{\R^3}&\bb{\np{1+\snorm{x-y}}\np{1+\wakefct{x-y}}}^{-2}
\np{1+\snorm{y}}^{-A}\np{1+\wakefct{y}}^{-B}\,\dy \\
&\qquad\qquad\leq C\np{1+\snorm{x}}^{-A}\np{1+\wakefct{x}}^{-B}\max\setL{1,\log\Bp{\frac{\snorm{x}}{1+\wakefct{x}}}}.
\end{aligned}
\]
\end{lem}
\begin{proof}
This is a consequence of the calculations in \cite[Section 2]{KracmarNovotnyPokorny2001}.
\end{proof}

\section{Main results}
\label{sec:ProofMainResults}

We consider weak solutions to \eqref{sys:NStp_Decay}
in the following sense.

\begin{defn}\label{def:WeakSolution_NStp}
Let $f\in\LRloc{1}(\torus\times\R^3)^3$.
A function $\uvel\in\LRloc{1}(\torus\times\R^3)^3$ 
is called \emph{weak solution} to \eqref{sys:NStp_Decay}
if
\begin{enumerate}[i.]
\item
$\uvel\in\LR{2}(\torus;\DSRNsigma{1}{2}(\R^3))$,
\label{item:WeakSolution_NStp_L2}
\item
$\projcompl\uvel\in\LR{\infty}(\torus;\LR{2}(\R^3))^3$,
\label{item:WeakSolution_NStp_LInfty}
\item 
the identity
\[
\int_{\torus\times\R^3}\bb{-\uvel\cdot\partial_t\varphi
+\grad\uvel:\grad\varphi
-\rey\partial_1\uvel\cdot\varphi
+\np{\uvel\cdot\grad\uvel}\cdot\varphi}\,\dtx
=\int_{\torus\times\R^3} f\cdot\varphi\,\dtx
\]
holds for all test functions $\varphi\in\CRcisigma(\torus\times\R^3)$.
\label{item:WeakSolution_NStp_WeakFormulation}
\end{enumerate}
\end{defn}

\begin{rem}
The existence of a weak solution with the above properties 
has been shown in \cite[Theorem 6.3.1]{Kyed_habil}
for any $f\in\LR{2}(\torus;\DSRN{-1}{2}(\R^3))^3$.
Therefore, this class seems to be a natural outset for further investigation.
Nevertheless, at first glance, instead of \ref{item:WeakSolution_NStp_LInfty}
one would expect the condition $\uvel\in\LR{\infty}(\torus;\LR{2}(\Omega))^3$,
which naturally appears for weak solutions to the Navier--Stokes initial-value problem.
However, this property cannot be expected for general time-periodic data $f$. 
As was shown by \textsc{Kyed} \cite[Theorem 5.2.4]{Kyed_habil},
for smooth data $f\in\CRci(\torus\times\R^3)^3$ 
one has $\uvel\in\LR{\infty}(\torus;\LR{2}(\R^3))^3$ 
if and only if $\int_{\torus\times\R^3} f\,\dxt=0$.
An analogous property
was established by \textsc{Finn} \cite{Finn1960}
for the corresponding steady-state problem.
\end{rem}

As our main result, we establish the following asymptotic expansions.

\begin{thm}\label{thm:VelocityExpansion}
Let $\rey\neq 0$ and $f\in\CRci(\torus\times\R^3)^3$, and let $\uvel$ be a 
weak time-periodic solution to \eqref{sys:NStp_Decay} in the sense of 
Definition \ref{def:WeakSolution_NStp}, which satisfies
\begin{equation}
\label{el:VelocityAdditionalIntegrability}
\exists r\in(5,\infty): \quad \projcompl\uvel\in\LR{r}(\torus\times\R^3)^3.
\end{equation}
Then
\begin{align}
\proj\uvel(x)
&=\fundsolvelss(x) \cdot \int_{\Omega} \proj f(y)\,\dy + \restterm_0(x), 
\label{eq:VelocityExpansion_ss}\\
\projcompl\uvel(t,x)
&=\fundsolvelpp(\cdot,x)\ast_\torus\int_{\Omega} \projcompl f(\cdot,y)\,\dy + \restterm_\perp(t,x)
\label{eq:VelocityExpansion_pp}
\end{align}
such that there exists $C>0$ such that for all $t\in\torus$ and $\snorm{x}\geq 4$ it holds
\begin{align}
\snorm{\restterm_0(x)}
&\leq C \bb{\snorm{x}\bp{1+\wakefct{\rey x}}}^{-3/2}\log\snorm{x},
\label{est:VelocityExpansion_resttermss_fct}
\\
\snorm{\grad\restterm_0(x)}
&\leq C \bb{\snorm{x}\bp{1+\wakefct{\rey x}}}^{-2}\max\setL{1,\log\Bp{\frac{\snorm{x}}{1+\wakefct{\rey x}}}},
\label{est:VelocityExpansion_resttermss_grad}
\\
\snorm{\restterm_\perp(t,x)}
&\leq C\snorm{x}^{-4},
\label{est:VelocityExpansion_resttermpp_fct}
\\
\snorm{\grad\restterm_\perp(t,x)}
&\leq C\snorm{x}^{-9/2}\np{1+\wakefct{\rey x}}^{-1/2}.
\label{est:VelocityExpansion_resttermpp_grad}
\end{align}
In particular, 
\begin{equation}
\uvel(t,x)
=\fundsolvelss(x) \cdot \int_\torus\int_{\Omega} f(t,y)\,\dy\dt + \restterm(t,x) 
\label{eq:VelocityExpansion_tp}\\
\end{equation}
with 
\begin{align}
\snorm{\restterm(t,x)}
&\leq C \bb{\snorm{x}\bp{1+\wakefct{\rey x}}}^{-3/2}\log\snorm{x},
\label{est:VelocityExpansion_resttermtp_fct}
\\
\snorm{\grad\restterm(t,x)}
&\leq C \bb{\snorm{x}\bp{1+\wakefct{\rey x}}}^{-2}
\max\setL{1,\log\Bp{\frac{\snorm{x}}{1+\wakefct{\rey x}}}}.
\label{est:VelocityExpansion_resttermtp_grad}
\end{align}
\end{thm}

\begin{rem}\label{rem:AsymptoticExpansion_SerrinCondition}
As explained in \cite{GaldiKyed_TPSolNS3D_AsymptoticProfile},
assumption \eqref{el:VelocityAdditionalIntegrability}
merely appears for technical reasons.
It ensures additional local regularity
but does not improve spatial decay of the solution.
\end{rem}

One main observation is that the asymptotic behavior of $\uvel$ and $\grad\uvel$
for $\snorm{x}\to\infty$ is governed by 
the time-periodic Oseen fundamental solution $\fundsolvel$.
In particular, the purely periodic part $\projcompl\uvel$ decays faster
than the steady-state part $\proj\uvel$. 
As a direct consequence of Theorem \ref{thm:VelocityExpansion},
we obtain the following pointwise estimates,
which we shall derive as intermediate results
on the way to a proof of Theorem \ref{thm:VelocityExpansion}.

\begin{thm}\label{thm:VelocityDecay}
Under the assumptions of Theorem \ref{thm:VelocityDecay}
there is $C>0$
such that for all $t\in\torus$ and $x\in\R^3$
the function $\uvel$ satisfies
\begin{align}
\snorm{\proj\uvel(x)}
&\leq C\bb{\bp{1+\snorm{x}}\bp{1+\wakefct{\rey x}}}^{-1}, 
\label{est:VelocityDecay_fct_ss}
\\
\snorm{\grad\proj\uvel(x)}
&\leq C\bb{\bp{1+\snorm{x}}\bp{1+\wakefct{\rey x}}}^{-\frac{3}{2}},
\label{est:VelocityDecay_grad_ss}
\\
\snorm{\projcompl\uvel(t,x)}
&\leq C\bp{1+\snorm{x}}^{-3},
\label{est:VelocityDecay_fct_pp}
\\
\snorm{\grad\projcompl\uvel(t,x)}
&\leq C\bp{1+\snorm{x}}^{-4}.
\label{est:VelocityDecay_grad_pp}
\end{align}
\end{thm}

In order to prove these theorems, we recall the following regularity result.

\begin{lem}\label{lem:RegularityNS}
Let $\uvel$ be a weak solution as in Theorem \ref{thm:VelocityExpansion}.
Then $\uvel\in\CRi(\torus\times\R^3)^3$ and
\begin{align*}
\forall r\in(1,\infty), \,q\in(1,2): &\ \ 
\grad^2\proj\uvel\in\LR{r}(\R^3),\ 
\grad\proj\uvel\in\LR{4q/(4-q)}(\R^3), \
\proj\uvel\in\LR{2q/(2-q)}(\R^3), \\
\forall q\in(1,\infty): &\ \
\projcompl\uvel\in\LR{q}(\torus;\WSR{2}{q}(\R^3))\cap\WSR{1}{q}(\torus;\LR{q}(\R^3),
\end{align*}
and there is a pressure function $\upres\in\CRi\np{\torus\times\R^3}$ such that 
\eqref{sys:NStp_Decay} is satisfied pointwise.
\end{lem}

\begin{proof}
We refer to \cite[Lemma 5.1]{GaldiKyed_TPSolNS3D_AsymptoticProfile}.
\end{proof}

We also need a uniqueness statement 
for solutions to the linear problem \eqref{sys:NSlintp_TPFS}.

\begin{lem}\label{lem:NSlintp_Uniqueness}
Let $\np{\uvel,\upres}\in\TDR(\grp)^{3+1}$ be a solution 
to \eqref{sys:NSlintp_TPFS} for the right-hand side $f=0$.
Then, $\proj\uvel$ is a polynomial in each component and $\projcompl\uvel=0$.
\end{lem}

\begin{proof}
An application of the Fourier transform $\FT_\grp$ on $\grp$ to
\eqrefsub{sys:NSlintp_TPFS}{1}
yields
\[
\bp{i\perft k + \snorm{\xi}^2-i\rey\xi_1}\ft\uvel+i\xi\ft\upres=0
\]
with $\ft\uvel\coloneqq\FT_\grp\nb{\uvel}$ and $\ft\upres\coloneqq\FT_\grp\nb{\upres}$.
Multiplying this equation with $i\xi$ and using $\Div\uvel=0$, 
we obtain $-\snorm{\xi}^2\ft\upres=0$, so that $\supp\ft\upres\subset\Z\times\set{0}$.
Then, the above equation yields
\[
\supp\bb{\np{i\perft k + \snorm{\xi}^2-i\rey\xi_1}\ft\uvel}
=\supp\bb{-i\xi\ft\upres}\subset\Z\times\set{0}.
\]
Because the only zero of $\np{k,\xi}\mapsto\np{i\perf k + \snorm{\xi}^2-i\rey\xi_1}$
is $(k,\xi)=(0,0)$, we conclude 
$\supp\ft\uvel\subset\set{(0,0)}$.
Thus we obtain $\projcompl\uvel=0$ and that $\proj\uvel$ is a polynomial.
\end{proof}

These lemmas enable us to derive the following representation formulas.

\begin{prop}\label{prop:RepresentationVelocity}
Let $\uvel$ be a weak solution as in Theorem \ref{thm:VelocityExpansion}.
Then 
\begin{align}\label{eq:RepresentationVelocity_uGradu}
\D_x^\alpha\uvel=\D_x^\alpha\fundsolvel\ast\nb{f-\uvel\cdot\grad\uvel}
\end{align}
for all $\alpha\in\N_0^3$ with $\snorm{\alpha}\leq1$.
In particular,
$\vvel\coloneqq\proj\uvel$ and $\wvel\coloneqq\projcompl\uvel$ satisfy 
\begin{align}
\D_x^\alpha\vvel&=\D_x^\alpha\fundsolvelss\ast\bb{\proj f - \vvel\cdot\grad\vvel-\proj\np{\wvel\cdot\grad\wvel}},
\label{eq:RepresentationVelocity_uGradu_ss}
\\
\D_x^\alpha\wvel&=\D_x^\alpha\fundsolvelpp\ast
\bb{\projcompl f - \vvel\cdot\grad\wvel - \wvel\cdot\grad\vvel-\projcompl\np{\wvel\cdot\grad\wvel}}.
\label{eq:RepresentationVelocity_uGradu_pp}
\end{align}
Moreover, we have%
\footnote{
Here we set 
$\np{\grad\fundsolvel\ast\Uvel}_j\coloneqq\partial_m \fundsolveljl\ast\Uvel_{jm}$
for an $\R^{3\times3}$-valued function $\Uvel$.
}
\begin{align}
\uvel&=\fundsolvel\ast f - \grad\fundsolvel\ast\np{\uvel\otimes\uvel},
\label{eq:RepresentationVelocity_uTensoru}
\\
\vvel&=\fundsolvelss\ast\proj f 
-\grad\fundsolvelss\ast\bb{\vvel\otimes\vvel+\proj\np{\wvel\otimes\wvel}},
\label{eq:RepresentationVelocity_uTensoru_ss}
\\
\wvel&=\fundsolvelpp\ast\projcompl f 
-\grad\fundsolvelpp\ast\bb{\vvel\otimes\wvel +\wvel\otimes\vvel+\projcompl\np{\wvel\otimes\wvel}}.
\label{eq:RepresentationVelocity_uTensoru_pp}
\end{align}
\end{prop}

\begin{proof}
From Lemma \ref{lem:RegularityNS} we conclude
$\uvel\cdot\grad\uvel\in\LR{q}(\torus\times\R^3)$ for all $q\in(1,\infty)$.
Therefore,
$\Uvel\coloneqq\fundsolvel\ast\np{f-\uvel\cdot\grad\uvel}$ is well defined as 
a classical convolution integral,
and we have
$\partial_j\Uvel=\partial_j\fundsolvel\ast\np{f-\uvel\cdot\grad\uvel}$ for $j=1,2,3$ by the 
dominated convergence theorem.
Since both $\Uvel$ and $\uvel$ satisfy the time-periodic Oseen system \eqref{sys:NSlintp_TPFS}
for suitable pressure functions $\upres$,
Lemma \ref{lem:NSlintp_Uniqueness} implies
$\projcompl\uvel=\projcompl\Uvel$ and that $\proj\uvel-\proj\Uvel$ is a polynomial
in each component.
With Young's inequality we obtain
$\proj\Uvel\in\LR{6}(\R^3)$
since $\fundsolvelss\in\LR{12/5}(\R^3)$ by \cite[Lemma 5.4]{GaldiKyed_TPSolNS3D_AsymptoticProfile}. 
Hence, $\proj\uvel-\proj\Uvel\in\LR{6}(\R^3)$.
This leads to $\proj\uvel=\proj\Uvel$ and thus $\uvel=\Uvel$, 
which yields \eqref{eq:RepresentationVelocity_uGradu}.
The remaining formulas now follow from
\begin{align*}
\vvel
&=\proj\uvel
=\np{\fundsolvelss\otimes\onedist_\torus}\ast\bb{f-\uvel\cdot\grad\uvel}
=\fundsolvelss\ast\bb{\proj\np{f-\uvel\cdot\grad\uvel}}, 
\\
\wvel
&=\projcompl\uvel
=\fundsolvelpp\ast\bb{f-\uvel\cdot\grad\uvel}
=\fundsolvelpp\ast\bb{\projcompl \np{f-\uvel\cdot\grad\uvel}}
\end{align*}
together with the identity $\uvel\cdot\grad\uvel=\Div\np{\uvel\otimes\uvel}$ and integration by parts.
\end{proof}

Based on these formulas, we can now prove Theorem \ref{thm:VelocityDecay}
and Theorem \ref{thm:VelocityExpansion}.

\begin{proof}[Proof of Theorem \ref{thm:VelocityDecay}]
We split $\uvel=\vvel+\wvel$ into steady-state part $\vvel\coloneqq\proj\uvel$ 
and purely periodic part $\wvel\coloneqq\projcompl\uvel$.
By \cite[Theorem 2.2]{GaldiKyed_TPSolNS3D_AsymptoticProfile} we have
\eqref{eq:VelocityExpansion_tp} with $\snorm{\restterm(t,x)}\leq C\snorm{x}^{-5/4}$.
In virtue of \eqref{est:fundsolss_Decay} and $\uvel\in\CRi(\torus\times\R^3)^3$, this implies 
\begin{align}
\snorm{\vvel(x)}
&\leq C \np{1+\snorm{x}}^{-1}\bp{1+\wakefct{\rey x}}^{-1/4}, 
\label{est:VelocityDecayFromAsymptExp_ss_improved}
\\
\snorm{\wvel(t,x)}
&\leq C \np{1+\snorm{x}}^{-5/4}
\label{est:VelocityDecayFromAsymptExp_pp_improved}
\end{align}
for all $t\in\torus$ and $x\in\R^3$.
This leads to
\[
\snorml{\vvel\otimes\vvel+\proj\nb{\wvel\otimes\wvel}}(x)
\leq C \np{1+\snorm{x} }^{-2}\np{1+\wakefct{\rey x}}^{-1/2}.
\]
Therefore, \eqref{eq:RepresentationVelocity_uTensoru_ss}, \eqref{est:LeadingTermss}
and Theorem \ref{thm:ConvFundsolss} yield
\[
\snorm{\vvel(x)}
\leq\snorml{\fundsolvelss\ast \proj f}(x) 
+ \snorml{\grad\fundsolvelss\ast\bb{\vvel\otimes\vvel+\proj\nb{\wvel\otimes\wvel}}}(x)
\leq C \bb{\np{1+\snorm{x}}\np{1+\wakefct{\rey x}}}^{-1},
\]
which is the desired estimate \eqref{est:VelocityDecay_fct_ss}.
Now \eqref{est:VelocityDecay_fct_ss} together with \eqref{est:VelocityDecayFromAsymptExp_pp_improved}
leads to
\begin{equation}\label{est:VelocityDecay_fct_pp01}
\snorml{\vvel\otimes\wvel+\wvel\otimes\vvel+\projcompl\nb{\wvel\otimes\wvel}}(t,x)
\leq C \np{1+\snorm{x}}^{-9/4}.
\end{equation}
Therefore, \eqref{eq:RepresentationVelocity_uTensoru_pp}, \eqref{est:LeadingTermpp}
and Theorem \ref{thm:ConvFundsolpp} imply
\[
\begin{aligned}
\snorm{\wvel(t,x)}
&\leq\snorml{\fundsolvelpp\ast \projcompl f}(t,x) 
+ \snorml{\grad\fundsolvelpp\ast\bb{\vvel\otimes\wvel
+\wvel\otimes\vvel
+\projcompl\nb{\wvel\otimes\wvel}}}(t,x)\\
&\leq C \bp{\np{1+\snorm{x}}^{-3}+\np{1+\snorm{x}}^{-9/4}}
\leq C \np{1+\snorm{x}}^{-9/4}.
\end{aligned}
\]
Using this estimate and \eqref{est:VelocityDecay_fct_ss} again, we conclude
\begin{equation}\label{est:VelocityDecay_fct_pp02}
\snorml{\vvel\otimes\wvel+\wvel\otimes\vvel+\projcompl\nb{\wvel\otimes\wvel}}(t,x)
\leq C \np{1+\snorm{x}}^{-13/4}.
\end{equation}
Repeating the above argument with \eqref{est:VelocityDecay_fct_pp02} instead of
\eqref{est:VelocityDecay_fct_pp01},
we end up with \eqref{est:VelocityDecay_fct_pp}.

Now let us turn to the estimates of $\grad\uvel$.
Due to $\uvel\in\CRi(\torus\times\R^3)$, 
it suffices to consider $\snorm{x}\geq 2$.
Let $R\coloneqq\snorm{x}/2\geq1$.
By Proposition \ref{prop:RepresentationVelocity} we have
\[
\partial_j\vvel=\partial_j\fundsolvelss\ast\proj f-I,
\qquad
\partial_j\wvel=\partial_j\fundsolvelpp\ast\projcompl f - J
\]
with
\[
\begin{aligned}
I
\coloneqq
I_1+I_2
&\coloneqq
\partial_j\fundsolvelss\ast\bb{\vvel\cdot\grad\vvel}
+\partial_j\fundsolvelss\ast\bb{\proj\nb{\wvel\cdot\grad\wvel}}, 
\\
J
\coloneqq J_1+J_2+J_3
&\coloneqq
\partial_j\fundsolvelpp\ast\bb{\vvel\cdot\grad\wvel}
+\partial_j\fundsolvelpp\ast\bb{\wvel\cdot\grad\vvel}
+\partial_j\fundsolvelpp\ast\bb{\projcompl\nb{\wvel\cdot\grad\wvel}}.
\end{aligned}
\]
We estimate these terms separately.
Clearly, $\snorm{I_1}\leq I_{11}+I_{12}$ with
\begin{align*}
I_{11}(x)&\coloneqq
\int_{\ball_{R}} \snorm{\partial_j\fundsolvelss(x-y)}\snorm{\vvel(y)}\snorm{\grad\vvel(y)} \,\dy, 
\\
I_{12}(x)&\coloneqq
\int_{\ball^{R}} \snorm{\partial_j\fundsolvelss(x-y)}\snorm{\vvel(y)}\snorm{\grad\vvel(y)} \,\dy.
\end{align*}
Since $\snorm{y}\leq R$ implies $\snorm{x-y}\geq\snorm{x}/2=R\geq 1$, 
the pointwise estimate \eqref{est:fundsolss_Decay} implies
\begin{align*}
I_{11}(x)
&\leq\int_{\ball_{R}} \bb{\np{1+\snorm{x-y}}\np{1+\wakefct{x-y}}}^{-3/2}\snorm{\vvel(y)}\snorm{\grad\vvel(y)} \,\dy\\
&\leq C \np{1+\snorm{x}}^{-3/2}\norm{\vvel}_{3}\norm{\grad\vvel}_{\frac{3}{2}}
\leq C \np{1+\snorm{x}}^{-3/2}
\end{align*}
in view of Lemma \ref{lem:RegularityNS},
and $\grad\fundsolvelss\in\LR{17/12}(\R^3)$ 
(see \cite[Lemma 5.4]{GaldiKyed_TPSolNS3D_AsymptoticProfile})
and Lemma \ref{lem:RegularityNS}
yield
\[
I_{12}(x)
\leq C \norm{\partial_j\fundsolvelss}_{\frac{17}{12}}
\norm{\grad\vvel}_{\frac{17}{5}}\norm{\vvel}_{\LR{\infty}(\ball^R)}
\leq C \np{1+\snorm{x}}^{-1}
\]
by \eqref{est:VelocityDecay_fct_ss}.
We thus deduce $\snorm{I_1(x)}\leq C \np{1+\snorm{x}}^{-1}$.
For $I_2$ we proceed similarly to obtain $\snorm{I_2(x)}\leq\np{1+\snorm{x}}^{-3/2}$.
From these estimates and \eqref{est:LeadingTermDiffss}, 
we conclude
\begin{equation}\label{est:DecayGradVelss_1}
\snorm{\grad\vvel(x)}\leq C \np{1+\snorm{x}}^{-1}.
\end{equation}
Now let us turn towards $\grad\wvel$. 
As above, we split $J_1$ and estimate $\snorm{J_1}\leq J_{11}+J_{12}$ with
\begin{align*}
J_{11}(t,x)
&\coloneqq\int_\torus\int_{\ball_{R}}\snorm{\partial_j\fundsolvelpp(t-s,x-y)}
\snorm{\vvel(y)}\snorm{\grad\wvel(s,y)}\,\dy\ds, \\
J_{12}(t,x)
&\coloneqq\int_\torus\int_{\ball^{R}}\snorm{\partial_j\fundsolvelpp(t-s,x-y)}
\snorm{\vvel(y)}\snorm{\grad\wvel(s,y)}\,\dy\ds.
\end{align*}
By H\"older's inequality in space and time, from \eqref{est:tpfundsol_ComplPointwiseEst}
we obtain
\[
J_{11}(t,x)
\leq C \Bp{\int_{\ball_{R}} \snorm{x-y}^{-8}\,\dy}^\half
\norm{\vvel}_4\norm{\grad\wvel}_4
\leq C \snorm{x}^{-5/2}
\]
due to Lemma \ref{lem:RegularityNS}.
Moreover, H\"older's inequality and \eqref{est:VelocityDecay_fct_ss}
lead to
\[
J_{12}(t,x)
\leq C \norm{\partial_j\fundsolvelpp}_1 
\norm{\vvel}_{\LR{\infty}(\ball^R)}\norm{\grad\wvel}_{\infty}
\leq C \snorm{x}^{-1}
\]
because $\grad\fundsolvelpp\in\LR{1}(\torus\times\R^3)$ 
by \eqref{est:tpfundsol_ComplSummabilityGradient}
and $\grad\wvel\in\LR{\infty}(\torus\times\R^3)$ by Lemma \ref{lem:RegularityNS} 
and Sobolev embeddings.
In a similar fashion, we can use \eqref{est:VelocityDecay_fct_pp} to estimate 
$J_2$ and $J_3$ and obtain
\begin{align*}
\snorm{J_2(t,x)}
&\leq C \bp{\snorm{x}^{-\frac{5}{2}}\norm{\wvel}_4\norm{\grad\vvel}_4
+\snorm{x}^{-3}\norm{\partial_j\fundsolvelpp}_{\frac{9}{8}}\norm{\grad\vvel}_9}
\leq C \snorm{x}^{-\frac{5}{2}}, 
\\
\snorm{J_3(t,x)}
&\leq C \bp{\snorm{x}^{-\frac{5}{2}}\norm{\wvel}_4\norm{\grad\wvel}_4
+\snorm{x}^{-3}\norm{\partial_j\fundsolvelpp}_{1}\norm{\grad\wvel}_\infty}
\leq C \snorm{x}^{-\frac{5}{2}}.
\end{align*}
Collecting the above estimates and combining them with \eqref{est:LeadingTermpp}, 
we end up with
\begin{equation}\label{est:DecayGradVelpp_1}
\snorm{\grad\wvel(t,x)}\leq C \np{1+\snorm{x}}^{-1}.
\end{equation}

From \eqref{est:VelocityDecay_fct_ss}, \eqref{est:DecayGradVelss_1},
\eqref{est:VelocityDecay_fct_pp} and \eqref{est:DecayGradVelpp_1} we now conclude
\[
\snorml{\vvel(x)\cdot\grad\vvel(x)+\proj\nb{\wvel\cdot\grad\wvel}(x)}
\leq C \np{1+\snorm{x}}^{-2}\np{1+\wakefct{\rey x}}^{-1/2},
\]
so that
\[
\snorm{I(x)}
\leq C \np{1+\snorm{x}}^{-5/4}\np{1+\wakefct{\rey x}}^{-3/4}
\]
by  Theorem \ref{thm:ConvFundsolss}.
Together with \eqref{est:LeadingTermss} we thus obtain
\[
\snorm{\grad\vvel(x)}\leq C \np{1+\snorm{x}}^{-5/4}\np{1+\wakefct{\rey x}}^{-3/4},
\]
so that from \eqref{est:VelocityDecay_fct_ss},
\eqref{est:VelocityDecay_fct_pp} and \eqref{est:DecayGradVelpp_1} we deduce 
\[
\snorml{\vvel(x)\cdot\grad\vvel(x)+\proj\nb{\wvel\cdot\grad\wvel}(x)}
\leq C \np{1+\snorm{x}}^{-9/4}\np{1+\wakefct{\rey x}}^{-7/4}.
\]
By another application of Theorem \ref{thm:ConvFundsolss} 
and combination with \eqref{est:LeadingTermss}, we arrive at
\eqref{est:VelocityDecay_grad_ss}.

For the derivation of \eqref{est:VelocityDecay_grad_pp}
we proceed with a similar bootstrap argument.
From \eqref{est:VelocityDecay_fct_ss}, \eqref{est:VelocityDecay_grad_ss}, 
\eqref{est:VelocityDecay_fct_pp} and \eqref{est:DecayGradVelpp_1} we deduce
\[
\snorml{\vvel\cdot\grad\wvel+\wvel\cdot\grad\vvel+\projcompl\bb{\wvel\cdot\grad\wvel}}(t,x)
\leq C \np{1+\snorm{x}}^{-2},
\]
so that Theorem \ref{thm:ConvFundsolpp} implies
$\snorm{J(t,x)}\leq C \np{1+\snorm{x}}^{-2}$.
Combining this with \eqref{est:LeadingTermpp},
we conclude
\begin{equation}\label{est:DecayGradVelpp_2}
\snorm{\grad\wvel(t,x)}\leq C \np{1+\snorm{x}}^{-2}.
\end{equation}
We now repeat this argument with \eqref{est:DecayGradVelpp_2} instead of 
\eqref{est:DecayGradVelpp_1},
which leads to an improved decay rate for $\grad\wvel$.
Iterating this procedure, we finally arrive at \eqref{est:VelocityDecay_grad_pp}. 
\end{proof}

\begin{proof}[Proof of Theorem \ref{thm:VelocityExpansion}]
We keep the notation from the previous proof.
We have
\begin{equation}\label{est:VelocityExpansion_uuss}
\snorml{\vvel\otimes\vvel+\proj\nb{\wvel\otimes\wvel}}(x)
\leq  C \bb{\np{1+\snorm{x}} \np{1+\wakefct{\rey x}}}^{-2}
\end{equation}
by Theorem \ref{thm:VelocityDecay},
which, by Theorem \ref{thm:ConvFundsolss}, implies
\[
\snormL{\partial_j\fundsolvelss\ast\bb{\vvel\otimes\vvel+\proj\nb{\wvel\otimes\wvel}}(x)}
\leq C \bb{\np{1+\snorm{x}}\bp{1+\wakefct{\rey x}}}^{-3/2}\log\snorm{x}.
\]
In virtue of the representation formula \eqref{eq:RepresentationVelocity_uTensoru_ss}
and 
the identity 
\[
\restterm_0(x)
=\proj\restterm(x)
=\vvel(x)-\fundsolvelss(x) \int_{\Rn} \proj f(y)\,\dy,
\] 
this estimate and \eqref{est:LeadingTermDiffss} 
imply \eqref{est:VelocityExpansion_resttermss_fct}.
Moreover, by Theorem \ref{thm:VelocityDecay} we have
\[
\snorml{\vvel\otimes\wvel+\wvel\otimes\vvel+\projcompl\nb{\wvel\otimes\wvel}}(t,x)
\leq C \np{1+\snorm{x}}^{-4},
\]
so that
\[
\snorml{\partial_j\fundsolvelpp\ast
\bb{\vvel\otimes\wvel+\wvel\otimes\vvel+\projcompl\nb{\wvel\otimes\wvel}}(t,x)}
\leq C\np{1+\snorm{x}}^{-4}
\]
by Theorem \ref{thm:ConvFundsolpp}.
Now \eqref{est:VelocityExpansion_resttermpp_fct} 
is a consequence of this estimate and \eqref{est:LeadingTermDiffpp}.

To show \eqref{est:VelocityExpansion_resttermss_grad}, 
at first observe that Theorem \ref{thm:VelocityDecay} implies 
\begin{equation}\label{est:VelocityExpansion_ugraduss}
\snorml{\vvel(x)\cdot\grad\vvel(x)+\proj\nb{\wvel\cdot\grad\wvel}(x)}
\leq C \bb{\np{1+\snorm{x}}\np{1+\wakefct{\rey x}}}^{-5/2}.
\end{equation}
Let
$\cutoff\in\CRci(\R^3)$ such that $\cutoff(x)=1$ for $\snorm{x}\leq 1$
and $\cutoff(x)=0$ for $\snorm{x}\geq 2$. 
We decompose 
\[
I=
\bb{\cutoff\partial_j\fundsolvelss}
\ast\bb{\vvel\cdot\grad\vvel+\proj\np{\wvel\cdot\grad\wvel}}
+ \bb{\np{1-\cutoff}\partial_j\fundsolvelss}
\ast\bb{\vvel\cdot\grad\vvel+\proj\np{\wvel\cdot\grad\wvel}}
\eqqcolon K_1+K_2.
\]
Then
\[
\snorm{K_1}
\leq C\int_{\ball_2(x)}\snorml{\partial_j\fundsolvelss(x-y)}
\bb{\np{1+\snorm{y}}\np{1+\wakefct{\rey y}}}^{-5/2}\,\dy
\]
by \eqref{est:VelocityExpansion_ugraduss}.
As in the proof of Lemma \ref{lem:LeadingTerm}, 
from $\snorm{x-y}\leq2\leq\snorm{x}/2$ we conclude $\snorm{y}\geq\snorm{x}/2\geq 2$ 
and $\np{1+4\snorm{\rey}}\np{1+\wakefct{\rey y}}\geq 1+\wakefct{\rey x}$. 
Since $\grad\fundsolvelss\in\LRloc{1}(\R^3)$, this implies
\[
\snorm{K_1}
\leq C \bb{\np{1+\snorm{x}}\np{1+\wakefct{\rey x}}}^{-5/2}
\int_{\ball_2(x)}\snorm{\partial_j\fundsolvelss(x-y)}\,\dy
\leq C \bb{\np{1+\snorm{x}}\np{1+\wakefct{\rey x}}}^{-5/2}.
\]
By integration by parts and \eqref{est:fundsolss_Decay} and 
\eqref{est:VelocityExpansion_uuss}, we further obtain
\begin{align*}
\snorm{K_2}
&\leq C\int_{\R^3}\snorm{1-\cutoff(x-y)}\,\snorml{\partial_j\grad\fundsolvelss(x-y)}\,
\snorml{\vvel\otimes\vvel+\proj\nb{\wvel\otimes\wvel}}(y)\, \dy\\
&\qquad\qquad+ C\int_{\R^3}\snorm{\grad\cutoff(x-y)}\snorml{\partial_j\fundsolvelss(x-y)}\,
\snorml{\vvel\otimes\vvel+\proj\nb{\wvel\otimes\wvel}}(y)\,\dy \\
&\leq C\int_{\ball^1(x)}\bb{\snorm{x-y}\np{1+\wakefct{\rey\np{x-y}}}}^{-2}
\bb{\np{1+\snorm{y}} \np{1+\wakefct{\rey y}}}^{-2}\, \dy\\
&\qquad\qquad+ C\int_{\ball_{1,2}(x)}\bb{\snorm{x-y}\np{1+\wakefct{\rey\np{x-y}}}}^{-3/2}
\bb{\np{1+\snorm{y}} \np{1+\wakefct{\rey y}}}^{-2}\,\dy.
\end{align*}
For the first integral we use Lemma \ref{lem:ConvEst_WakeFct_Two}, 
and for the second one we argue as for $K_1$ to deduce
\begin{align*}
\snorm{K_2}
&\leq C\bp{\np{1+\snorm{x}}\np{1+\wakefct{\rey x}}}^{-2}
\max\setL{1,\log\Bp{\frac{\snorm{x}}{1+\wakefct{\rey x}}}}\\
&\qquad
+C\bb{\np{1+\snorm{x}} \np{1+\wakefct{\rey x}}}^{-2}
\int_{\ball_{1,2}(x)}\bb{\snorm{x-y}\np{1+\wakefct{\rey\np{x-y}}}}^{-3/2}\,\dy \\
&\leq C\bp{\np{1+\snorm{x}}\np{1+\wakefct{\rey x}}}^{-2}
\max\setL{1,\log\Bp{\frac{\snorm{x}}{1+\wakefct{\rey x}}}}.
\end{align*}
Combining the estimates of $K_1$ and $K_2$ with
\eqref{est:LeadingTermDiffss},
from formula \eqref{eq:RepresentationVelocity_uGradu_ss}
we obtain \eqref{est:VelocityExpansion_resttermss_grad}.

Furthermore, Theorem \ref{thm:VelocityDecay} implies
\[
\snorml{\vvel\cdot\grad\wvel+\wvel\cdot\grad\vvel+\projcompl\bb{\wvel\cdot\grad\wvel}}(t,x)
\leq C \np{1+\snorm{x}}^{-9/2}\np{1+\wakefct{\rey x}}^{-3/2}.
\]
With an argument similar to before, we now deduce
\begin{align*}
\snorm{J(t,x)}
&\leq C\int_{\ball_1(x)}\snorm{\partial_j\fundsolvelpp(x-y)}
\np{1+\snorm{y}}^{-9/2}\np{1+\wakefct{\rey y}}^{-3/2}\,\dy 
\\
&\qquad\qquad\qquad+C\int_{\ball^1(x)}\snorm{\partial_j\fundsolvelpp(x-y)}
\np{1+\snorm{y}}^{-9/2}\np{1+\wakefct{\rey y}}^{-3/2}\,\dy
\\
&\leq C \Bp{\np{1+\snorm{x}}^{-9/2}\np{1+\wakefct{\rey x}}^{-3/2}
\norm{\partial_j\fundsolvelpp}_1
+\int_{\ball^1(x)}\snorm{x-y}^{-4}
\np{1+\snorm{y}}^{-9/2}\,\dy}\\
&\leq C\np{1+\snorm{x}}^{-9/2}\np{1+\wakefct{\rey x}}^{-1}
\end{align*}
where we used \eqref{est:tpfundsol_ComplSummabilityGradient}.
Combining this estimates with \eqref{est:LeadingTermDiffpp},
formula
\eqref{eq:RepresentationVelocity_uGradu_pp} 
implies \eqref{est:VelocityExpansion_resttermpp_grad}.

Finally, the asymptotic expansion \eqref{eq:VelocityExpansion_tp} with the asserted estimates of 
$\restterm(t,x)$ is a direct consequence of these results and 
the pointwise estimates of $\fundsolvelpp$ from \eqref{est:tpfundsol_ComplPointwiseEst}.
\end{proof}

\bibliographystyle{abbrv}

\smallskip\par\noindent
Weierstrass Institute for Applied Analysis and Stochastics\\ 
Mohrenstra\ss{}e 39, 10117 Berlin, Germany\\
Email: {\texttt{thomas.eiter@wias-berlin.de}}

\end{document}